\documentclass[reqno]{amsart}
\usepackage{hyperref}

\begin{document}

\title[\hfil Existence Results to a nonlinear $p(k)$-Laplacian difference equation] {Existence Results to a nonlinear $p(k)$-Laplacian difference equation}

\author[M. Khaleghi Moghadam \hfilneg]
{Mohsen Khaleghi Moghadam}

\address{Mohsen Khaleghi Moghadam \newline
Department of Basic Science, Sari Agricultural Sciences and natural
Resources University, 578 Sari, Iran\newline }
\email{mohsen.khaleghi@rocketmail.com \quad and \quad
m.khaleghi@sanru.ac.ir}
\author[M. Avci\hfilneg]
{Mustafa Avci}
\address{ Mustafa Avci \newline
Faculty of Economics and Administrative Sciences, Batman University,
Turkey  } \email{ avcixmustafa@gmail.com}


\begin{abstract}
In the present paper, by using variational method, the existence of non-trivial solutions to an anisotropic discrete non-linear problem involving $p(k)$-Laplacian operator with Dirichlet boundary condition is investigated. The main technical tools applied here are the two local minimum theorems for differentiable functionals given by Bonanno.
\end{abstract}

\maketitle \numberwithin{equation}{section}
\newtheorem{theorem}{Theorem}[section]
\newtheorem{proposition}[theorem]{Proposition}
\newtheorem{corollary}[theorem]{Corollary}
\newtheorem{lemma}[theorem]{Lemma}
\newtheorem{remark}[theorem]{Remark}
\newtheorem{example}[theorem]{Example}
\allowdisplaybreaks

\section{Introduction}
The main goal of the present paper is to establish the existence of non-trivial
solution for the following discrete anisotropic problem
\begin{equation}
\label{1}
\begin{cases}
-\Delta(w(k-1)|\Delta u(k-1)|^{p(k-1)-2}\Delta u(k-1))+q(k)|u(k)|^{p(k)-2}u(k)=\lambda f(k,u(k)), \\
u(0)=u(T+1)=0,
\end{cases}
\end{equation}
for any $k \in [1,T]$, where $T$ is a fixed positive integer, $[1,T]$ is the discrete interval $\{1,...,T\}$,  $f: [1,T]\times\mathbb{R}\to\mathbb{R}$ is a continuous function, $\lambda>0$ is a parameter and $w :[0,T]\to[1,\infty)$ is a fix function such that
\begin{equation*}
w^{-}:=\min_{k\in \lbrack 0,T]}w(k),\ \ \ \ w^{+}:=\max_{k\in
\lbrack 0,T]}w(k),
\end{equation*}%
 and $\Delta u(k)=u(k+1)-u(k)$ is the forward difference operator and the function $p :[0,T+1]\to[2,\infty)$ is bounded, we denote for short
 $$p^+:=\max_{k\in[0,T+1]}p(k) \quad \text{and} \quad p^-:=\min_{k\in[0,T+1]}p(k),$$ and the function $q :[0,T+1]\to[1,\infty)$ is bounded such that
\begin{equation*}
q^{-} :=\min_{k\in \lbrack 1,T+1]}q(k)\geq1,\ \ \ \
q^{+}:=\max_{k\in \lbrack
1,T+1]}q(k) \\
\end{equation*}

We want to remark that problem \eqref{1} is the discrete variant of the variable exponent anisotropic problem
\begin{equation}\label{1-1}
\begin{cases}
-\sum_{i=1}^{N}\frac{\partial}{\partial x_i}(w_i(x)|\frac{\partial u}{\partial x_i}|^{p_i(x)-2}\frac{\partial u}{\partial x_i})+q(x)|u|^{p_i(x)-2} u=\lambda
 f(x,u), \quad x \in \Omega, \\
u=0, x \in \partial\Omega,
\end{cases}
\end{equation}
 where $\Omega\subset\mathbb{R}^N$, $N\geq 3$
is a bounded domain with smooth boundary, $f\in C(\overline{\Omega}\times\mathbb{R},\mathbb{R})$ is  given function that satisfy certain properties and $p_i(x)$, $w_i(x)\geq1$ and $q(x)\geq1$ are continuous functions on $\overline{\Omega}$ with $2 \leq p_i(x)$ for each $x\in \Omega$ and every
  $i\in\{1,2,\cdots,N\}$,
  $\lambda>0$ is real number.\\

The importance of difference equations arises from its applications to many different fields of research, such as
mechanical engineering, control systems, economics, social sciences, computer
science, physics, artificial or biological neural networks, cybernetics, ecology, to name a few. In this context, anisotropic discrete non-linear problems involving $p(k)$-Laplacian operator seem to have attracted a great deal of attention due to its usefulness of modelling some more complicated phenomenon such us fluid dynamics and nonlinear elasticity. We refer the reader to \cite{APO1,AHH,Avci,AvcPan,AH,BX,BSZ,BC1,CD2,CJ,GMW,HKh,HT,JZ,KhHH,KhH,LL,LG,R} and references therein, where they could find the  detailed background as well as many different approaches and techniques applied in the related area.

\bigskip

In this paper, based on two local minimum theorems, i.e. Theorem \ref{t1} and Theorem \ref{t1-1}) due to Bonanno \cite{B}, we obtain the exact intervals for the parameter $\lambda$, in which the problem \eqref{1} admits non-trivial solutions. In section 2, we recall the main tools (Theorem\ref{t1} and Theorem\ref{t1-1}) and give some basic knowledge. In Section 3, we state and prove our main results of the paper containing several theorems as well as corollaries. Finally, we prove a special case of the main result (Theorem \ref{t1.1}) and illustrate the results by giving concrete examples as applications to \eqref{1}.

\bigskip

At the very beginning, as an example, we give the following special case of our main results.
\begin{theorem}\label{t1.1}Let $T$ is a fixed positive integer. Assume that there exist two positive constants $c$ and $d$ such that

\begin{equation}\label{1-1-1}
    d^{3}<\left(\frac{3}{(T+2)(T+4)(2T+2)^{T+3}}\right)c^{T+4}<\frac{3}{(T+2)(T+4)}.
\end{equation}
 Let $g:\mathbb{R}\to\mathbb{R}$ be a non-negative continuous function such
that $\int_0^tg(s)ds<c_0(1+t^2)$ for any $t\in\mathbb{R}$ and some
$c_0>0$ and
\begin{equation}\label{1-1-2}
\frac{\int_0^cg(\xi)d\xi}{c^{T+4}}
<\left(\frac{3}{(T+4)(T+2)(2T+2)^{T+3}}\right)\frac{\int_0^dg(\xi)d\xi}{d^{3}}.
\end{equation}
Then, for each
$$\lambda\in\Lambda_{c,d}=\left]\frac{d^{3}(T+2)}{3T\int_0^dg(\xi)d\xi},\frac{3c^{T+4}-d^{3}(T+4)(T+2)(2T+2)^{T+3}}{
3T(T+4)(2T+2)^{T+3}\int_d^cg(\xi)d\xi}\right[, $$ the problem
$$\begin{cases}
-\Delta(|\Delta u(k-1)|^{k}\Delta u(k-1))+|u(k)|^{k+1}u(k)=\lambda g(u(k)), \quad k \in [1,T], \\
u(0)=u(T+1)=0,
\end{cases}$$
admits at least one positive solution in the space $\{u : [0,T+ 1]\to \mathbb{R} : u(0) =
u(T+1) =0 \}$.
\end{theorem}

\section{Preliminaries}

First, we give the following definition. For given a set $X$ and two
functionals  $\Phi,\ \Psi:X\to \mathbb{R}$, we defined the following
functions
\begin{equation}\label{2}
\beta(r_1,r_2)=\inf_{v\in \Phi^{-1} (]r_1,r_2[)}\frac{\sup_{u\in \Phi^{-1} (]r_1,r_2[)} \Psi(u)-\Psi(v)}{r_2-\Phi(v)},
\end{equation}
and
\begin{equation}\label{3}
\rho_1(r_1,r_2)=\sup_{v\in \Phi^{-1} (]r_1,r_2[)} \frac{\Psi(v)-\sup_{u\in \Phi^{-1} (]-\infty,r_1])} \Psi(u)}{\Phi(v)-r_1},
\end{equation}
for all $r_1,r_2 \in\mathbb{R}$, with $r_1<r_2$.\\
\begin{equation}\label{3}
\rho_2(r)=\sup_{v\in \Phi^{-1} (]r,\infty[)} \frac{\Psi(v)-\sup_{u\in \Phi^{-1} (]-\infty,r])} \Psi(u)}{\Phi(v)-r},
\end{equation}
for all $r\in\mathbb{R}$.\\
\begin{theorem} (\cite[Theorem 5.1, Proposition 2.1]{B}) \label{t1}
Let $X$ be a reflexive real Banach space,  $ \Phi:X \to \mathbb{R}$ be
a sequentially weakly lower semicontinuous, coercive and continuously G\^ateaux differentiable functional whose G\^ateaux derivative admits a
 continuous inverse on $X^*$ and  $ \Psi:X \to \mathbb{R}$ be a continuously G\^ateaux differentiable functional whose G\^ateaux derivative
  is compact. Put $I_\lambda=\Phi-\lambda\Psi$ and assume that there are $r_1,r_2 \in\mathbb{R}$, with $r_1<r_2$, such that
$$\beta(r_1,r_2)<\rho_1(r_1,r_2),$$
where $\beta$ and $\rho_1$ are given by (\ref{2}) and (\ref{3}). Then, for
each $\lambda\in \Lambda=]\frac{1}{\rho_1(r_1,r_2)},\frac{1}{\beta(r_1,r_2)}[$ there is $u_{0,\lambda}\in \Phi^{-1} (]r_1,r_2[)$
such that $I_\lambda(u_{0,\lambda})\leq I_\lambda(u)$ for all $u\in \Phi^{-1} (]r_1,r_2[)$
and $I'_\lambda(u_{0,\lambda})=0$.
\end{theorem}
\begin{theorem} (\cite[Theorem 5.5]{B}) \label{t1-1}
Let $X$ be a real Banach space,  $ \Phi:X \to \mathbb{R}$
be a  continuously G\^ateaux differentiable functional whose G\^ateaux derivative admits a
 continuous inverse on $X^*$ and  $ \Psi:X \to \mathbb{R}$ be a continuously G\^ateaux differentiable functional whose G\^ateaux derivative
  is compact. Fix $\inf_X\Phi<r<\sup_X\Phi$ and assume that
$$\rho_2(r)>0,$$
and for each $\lambda>\frac{1}{\rho_2(r)}$, the functional
$I_\lambda:=\Phi-\lambda\Psi$ is coercive. Then, for each
$\lambda\in ]\frac{1}{\rho_2(r)},+\infty[$ there is
$u_{0,\lambda}\in \Phi^{-1} ]r,+\infty[$ such that
$I_\lambda(u_{0,\lambda})\leq I_\lambda(u)$ for all $u\in \Phi^{-1}
]r,+\infty[$ and $I'_\lambda(u_{0,\lambda})=0$.
\end{theorem}
Theorems \ref{t1} is  a consequence of a local minimum theorem
\cite[Theorem 3.1]{B} which is a more general version of the Ricceri
Variational Principle (see \cite{R}) and Theorem \ref{t1-1} is a
consequence of \cite[Theorem 4.2]{B} which is a relevant variant of
\cite[Theorem 3.1]{B}.

  Let  $T\geq2$ be a fixed positive integer, $[1,T]$ denote a discrete interval $\{1,...,T\}$.
Define $T$-dimensional function space by
$$W:=\{u : [0,T+ 1]\to \mathbb{R} : u(0) = u(T+1)=0\},$$
 which is a Hilbert space under the norm
$$\|u\|=\left\{\sum^{T+1}_ {k=1}w(k-1) |\Delta u(k-1)|^{p^-} +q(k)|u(k)|^{p^-}\right\}^{1/p^-}.$$
Since $W$ is finite-dimensional, we can also define the following equivalent norm on $W$
$$\|u\|_+=\left\{\sum^{T+1}_ {k=1} w(k-1)|\Delta u(k-1)|^{p^+} +q(k)|u(k)|^{p^+}\right\}^{1/p^+}.$$
It is clear that by weighted H\"{o}lder inequality, one can conclude
\begin{equation}\label{in}
     K_0\|u\|\leq \|u\|_+\leq 2^{\frac{p^+-p^-}{p^+p^-}}K_0\|u\|,
\end{equation}
where,
$$K_0=
\Big\{(2T+2)\max\{w^+,q^+\}\Big\}^{\frac{p^--p^+}{p^+p^-}}.$$
 Now, let $\varphi: W \to \mathbb{R} $ be given by the formula
 \begin{equation}\label{1-3}
    \varphi(u):=\sum^{T+1}_ {k=1}\Big[ w(k-1)|\Delta u(k-1)|^{p(k-1)} +q(k)|u(k)|^{p(k)}\Big].
 \end{equation}
In the rest of the paper, the following lemma will be very useful.
\begin{lemma}\label{L1}
For any $u\in W$, there exists a positive constant $C_1$ such that
 \begin{equation}\label{x1}
    \|u\|<1\Rightarrow \|u\|_+^{p^+}\leq
    \varphi(u)\leq\|u\|^{p^-},
     \end{equation}
\begin{equation}\label{x11}
    \|u\|\geq1\Rightarrow \|u\|^{p^-}-C_1\leq \varphi(u)\leq \|u\|_+^{p^+}+C_1.
 \end{equation}
   where $C_1=(T+1)(w^++q^+)$.
 \end{lemma}
\begin{proof}
Let $u\in W$ be  fixed.
 By a similar approach argued in \cite{MRT2}, we
set
$$A^<:=\{k\in [0,T+1], |\Delta u(k)|<1\}, \quad B^<:=\{k\in [0,T+1], |u(k)|<1\},$$
$$A^\geq:=\{k\in [0,T+1], |\Delta u(k)|\geq1\}, \quad B^\geq:=\{k\in [0,T+1], |u(k)|\geq1\}.$$
 If $\|u\|<1$, then $A^\geq=B^\geq=\emptyset$ and $A^>=B^>=[0,T+1]$. It follows that $|\Delta u(k-1)|,
|u(k)|<1$ for each $k\in[1,T+1]$. Hence, we have
\begin{eqnarray*}
 \varphi(u)\leq \sum^{T+1}_ {k=1}\Big[w(k-1) |\Delta u(k-1)|^{p^-} +q(k)|u(k)|^{p^-}\Big]= \|u\|^{p^-}.
\end{eqnarray*}
\begin{eqnarray*}
\|u\|_+^{p^+ }=  \sum^{T+1}_ {k=1}\Big[w(k-1) |\Delta u(k-1)|^{p^+}
+q(k)|u(k)|^{p^+}\Big]\leq \varphi(u),
\end{eqnarray*}
which means that \eqref{x1} holds.  If $\|u\|\geq1$, we have
\begin{eqnarray*}
 &&\sum^{T+1}_ {k=1}\Big[ \frac{w(k-1)}{{p(k-1)}}|\Delta u(k-1)|^{p(k-1)}\Big]  \\
   &=& \Big[\Big( \sum_ {k\in A^<}+\sum_ {k\in A^\geq}\Big)w(k-1)|\Delta u(k-1)|^{p(k-1)}\Big]\\
&\leq& \Big[  \sum_{k\in A^<} w(k-1)|\Delta u(k-1)|^{p^-}+\sum_{k\in A^\geq} w(k-1)|\Delta u(k-1)|^{p^+}\Big]\\
   &=& \Big[\sum^{T+1}_ {k=1}w(k-1)|\Delta u(k-1)|^{{p^+}}\\
   &+& \sum_ {k\in A^<} \Big(w(k-1)(|\Delta u(k-1)|^{{p^-}}-|\Delta u(k-1)|^{p^+})\Big)\Big]\\
   &\leq&\Big[\sum^{T+1}_ {k=1}w(k-1)|\Delta u(k-1)|^{{p^+}}\\
   &+& w^+ \sum_ {k\in A^<} \Big(|\Delta u(k-1)|^{{p^-}}-|\Delta u(k-1)|^{p^+}\Big)\Big]\\
   &\leq& \Big[\sum^{T+1}_ {k=1}w(k-1)|\Delta u(k-1)|^{{p^+}}+ (T+1)w^+\Big].
\end{eqnarray*}
By a similar argument, it reads
\begin{eqnarray*}
 &&\sum^{T+1}_ {k=1}\Big[ \frac{w(k-1)}{{p(k-1)}}|\Delta u(k-1)|^{p(k-1)}\Big]  \\
   &=& \Big[\Big( \sum_ {k\in A^<}+\sum_ {k\in A^\geq}\Big)w(k-1)|\Delta u(k-1)|^{p(k-1)}\Big]\\
&\geq& \Big[  \sum_{k\in A^<} w(k-1)|\Delta u(k-1)|^{p^+}+\sum_{k\in A^\geq} w(k-1)|\Delta u(k-1)|^{p^-}\Big]\\
   &=& \Big[\sum^{T+1}_ {k=1}w(k-1)|\Delta u(k-1)|^{{p^-}}\\
   &-& \sum_ {k\in A^<} \Big(w(k-1)(|\Delta u(k-1)|^{{p^-}}-|\Delta u(k-1)|^{p^+})\Big)\Big]\\
   &\geq&\Big[\sum^{T+1}_ {k=1}w(k-1)|\Delta u(k-1)|^{{p^-}}\\
   &-& w^+ \sum_ {k\in A^<} \Big(|\Delta u(k-1)|^{{p^-}}-|\Delta u(k-1)|^{p^+}\Big)\Big]\\
   &\geq& \Big[\sum^{T+1}_ {k=1}w(k-1)|\Delta u(k-1)|^{{p^-}}- w^+(T+1)\Big].
\end{eqnarray*}
Combining the above inequalities we obtain
\begin{eqnarray*}
\Big[\sum^{T+1}_ {k=1}w(k-1)|\Delta u(k-1)|^{{p^-}}- w^+(T+1)\Big] &\leq&\sum^{T+1}_ {k=1}\Big[ \frac{w(k-1)}{{p(k-1)}}|\Delta u(k-1)|^{p(k-1)}\Big] \\
   &\leq& \Big[\sum^{T+1}_ {k=1}w(k-1)|\Delta u(k-1)|^{{p^+}}+ w^+(T+1)\Big].
\end{eqnarray*}
In the same manner we get
\begin{eqnarray*}
 \Big[\sum^{T+1}_ {k=1}q(k)|u(k)|^{p^-}-q^+(T+1)\Big]&\leq&\sum^{T+1}_ {k=1}\Big[\frac{q(k)}{{p(k)}}|u(k)|^{p(k)}\Big]  \\
   &\leq& \Big[\sum^{T+1}_ {k=1}q(k)|u(k)|^{p^+}+q^+(T+1)\Big]
\end{eqnarray*}
Combining the above double inequalities we obtain
\begin{eqnarray*}
 \|u\|^{p^-}-(T+1)(w^++q^+)\leq \varphi(u)\leq
\|u\|_+^{p^+}+(T+1)(w^++q^+).
\end{eqnarray*}

\end{proof}

Let $\Phi$ and $\Psi$ be as in the following
 \begin{eqnarray}\label{2-1}
  \Phi(u)&:=& \sum^{T+1}_ {k=1}\left[ \frac{w(k-1)}{{p(k-1)}}|\Delta u(k-1)|^{{p}(k-1)} +\frac{q(k)}{{p(k)}}|u(k)|^{p(k)}\right], \\
 \Psi(u) &:=& \sum_{k=1}^{T}F(k,u(k)),\label{2-2}
\end{eqnarray}
 where $F(k,t):=\int^t_ 0 f(k,\xi)d\xi$ for every
$(k,t)\in [1,T]\times \mathbb{R}$.\\
 In the sequel, we will use the following inequality
\begin{equation}\label{4}
    \|u\|_\infty:=\max_{ k\in[1,T]}|u(k)|\leq (2T+2)^{\frac{p^--1}{p^-}}\|u\|, \ \ \forall u\in W,
\end{equation}

To study the problem \eqref{1}, we consider the functional $I_{\lambda,\mu}:W\to\mathbb{R}$ defined by
\begin{eqnarray}\label{I}
 I_{\lambda}(u)&=& \sum^{T+1}_ {k=1}\left[ \frac{w(k-1)}{{p(k-1)}}|\Delta u(k-1)|^{p(k-1)} +\frac{q(k)}{{p(k)}}|u(k)|^{p(k)}\right]\nonumber\\
&-&\lambda
\sum_{k=1}^{T}F(k,u).
\end{eqnarray}
We want to remark that since problem $(1.1) $ is settled in a finite-dimensional Hilbert space $W$, it is not difficult to verify that the
functional $I_{\lambda }$ satisfies the regularity properties. Therefore $I_{\lambda }$ is of class $C^{1}$ on $W$ (see, e.g., \cite{JZ}) with the derivative
 \begin{eqnarray*}
 I'_{\lambda}(u)(v)&=& \sum_{k=1}^{T+1}\left[w(k-1)|\Delta u(k-1)|^{p(k-1)-2}\Delta u(k-1)\Delta v(k-1)\right]\\
 &+&\sum_{k=1}^{T}q(k)|u(k)|^{p(k)-2}u(k)v(k)\\
&-&\sum_{k=1}^{T}\left[\lambda f(k,u(k))\right]v(k),
 \end{eqnarray*}
 for all $u,v\in W$.
\begin{lemma}\label{l}
 The critical points of
$I_{\lambda}$ and the solutions of the problem \eqref{1} are exactly
equal.
\end{lemma}
\begin{proof}
Let $\overline{u}$ be a critical point of $I_{\lambda}$ in $W$.
Thus, for every $v\in W$, taking $v(0)=v(T+1)=0$ into account and
applying summation by parts, one has
\begin{eqnarray*}
0&=&I'_{\lambda}(\overline{u})(v) \\
&=&-\sum_{k=1}^{T}\Big[\Delta(w(k-1)|\Delta
\overline{u}(k-1)|^{p(k-1)-2}\Delta
\overline{u}(k-1)) \\
&-& q(k)|\overline{u}(k)|^{p(k)-2}\overline{u}(k)+\lambda f(k,\overline{u}(k))\Big]v(k).
\end{eqnarray*}
Since $v\in W$ is arbitrary, it reads
  $$-\Delta(w(k-1)|\Delta
\overline{u}(k-1)|^{p(k-1)-2}\Delta
\overline{u}(k-1)
+q(k)|
\overline{u}(k)|^{p(k)-2}
\overline{u}(k)=\lambda f(k,\overline{u}(k)),$$
for every $k \in [1,T]$. Therefore, $\overline{u}$ is a solution of \eqref{1}.
 So by  bearing in mind that $\overline{u}$ is arbitrary, we conclude that every  critical point of
the functional $I_{\lambda}$ in $W$, is  exactly a solution of the
problem \eqref{1}.\\
Vice versa, if $\overline{u}\in W$ be a solution of problem
\eqref{1}, by multiplying the difference equation in problem
\eqref{1} by $v(k)$ as an arbitrary element of $W$ and summing and
using the fact that
$$\sum_{k=1}^{T+1}w(k-1)\phi_{p(k-1)}(\Delta
\overline{u}(k-1))\Delta v(k-1)=-\sum_{k=1}^{T}\Delta \left(
w(k-1)\phi_{p(k-1)}(\Delta \overline{u}(k-1))\right)v(k),$$ we have
$I'_{\lambda,\mu}(\overline{u})(v)=0$, hence $\overline{u}$ is a
critical point for $I_{\lambda,\mu}$. Thus the viceversa holds and
the proof is completed.
\end{proof}


\section{Main Results}
First, put
$$A=\left(w(0)+w(T)+\sum_{k=1}^{T}q(k)\right)$$
and
$$K=(2T+2)^{\frac{1-p^+}{p^+}}\Big\{\max\{w^+,q^+\}\Big\}^{\frac{p^--p^+}{p^+p^-}}.$$
Moreover, for given two non-negative constants $c$ and $d$ with $\frac{1}{p^+}\left({cK}\right)^{p^+}\neq
\frac{d^{p^-}}{p^-}A$, define
$$a_d(c):=\frac{\sum^T _{k=1} \max_{ |\xi|\leq c} F(k,\xi)-\sum^T_{k=1} F(k,d)}{\frac{1}{p^+}\left({cK}\right)^{p^+}-\frac{d^{p^-}}{p^-}A}.$$

\begin{itemize}
\item[$(F1)$] There exist a constant $c_{0}>0$ and a function $\alpha:\mathbb{Z}[1,T]\rightarrow\lbrack 2,+\infty )$,
  with $\max_{k\in[1,T]}\alpha(k):=\alpha^{+}<p^{-}$, such that for all $(k,t)\in \mathbb{Z}[1,T]\times\mathbb{R}$,
\begin{equation*}
F(k,t)\leq c_{0}(1+|t|^{\alpha(k)})  .
\end{equation*}

\end{itemize}
 Now, we are ready to state our first main result as follows.
\begin{theorem}\label{t2}
Assume the condition $(F1)$ holds and assume that there exist a
non-negative constant $c_1$ and two positive constants $c_2$ and $d$
such that
%
\begin{equation}\label{in2}
    \left\{\frac{K}{A^{\frac{1}{{p^+}}}}\right\}c_1<d<
    \left(\frac{p^-K^{p^+}}{p^+A}\right)^{\frac{1}{p^-}}c_2^{\frac{p^+}{p^-}}<\left(\frac{p^-}{p^+A}\right)^{\frac{1}{p^-}},
\end{equation}
and  \\\\
$(F2)$ $\quad a_d(c_2)<a_d(c_1).$\\\\
Then for any $\lambda\in
]\frac{1}{a_d(c_1)},\frac{1}{a_d(c_2)}[$ the problem \eqref{1}
has at least one non-trivial solution $u_0\in W$.
\end{theorem}
\begin{proof}
Our aim is to apply Theorem \ref{t1} to problem \eqref{1}. To settle
the variational framework of problem \eqref{1}, take $X=W$, and put
$\Phi,\ \Psi$  as defined in \eqref{2-1} and \eqref{2-2},
respectively for every $u\in W$. Due to $(F2)$, the interval
$]\frac{1}{a_d(c_1)},\frac{1}{a_d(c_2)}[$ is non-empty. Therefore,
if we fix $\bar{\lambda}$ in this interval, we can write
$$I_{\bar{\lambda}}=\Phi-\bar{\lambda} \Psi.$$
Again, because $W$ is finite dimensional, an easy computation ensures that $\Phi$ and  $\Psi$ are of class $C^1$ on $W$ with the derivatives
\begin{eqnarray*}
  \Phi'(u)(v) &=&\sum_{k=1}^{T+1}w(k-1)|\Delta{u}(k-1)|^{p(k-1)-2}\Delta{u}(k-1)\Delta v(k-1) \\
   &+& \sum_{k=1}^{T}q(k)|u(k)|^{p(k)-2}u(k)v(k) \\
   &=&-\sum_{k=1}^{T}\Delta(w(k-1)|\Delta{u}(k-1)|^{p(k-1)-2}\Delta{u}(k-1))v(k)\\
   &+&\sum_{k=1}^{T}q(k)|u(k)|^{p(k)-2}u(k)v(k),
\end{eqnarray*}
and
$$\Psi '(u)(v)=\sum_{k=1}^{T}f(k,u(k))v(k),
$$
for all $u,v\in W$.  Hence $\Phi$ is sequentially weakly
semicontinuous functional. Also $\Phi$ is coercive. Indeed,  let
$u\in W$ be a fixed member with $\|u\|>1$.
   From \eqref{x11}, we have
   $$\Phi(u)\geq \frac{1}{p^+}\varphi(u)\geq\frac{\|u\|^{p^-}}{p^+}-\frac{C_1}{p^+}.$$
 Therefore  $\Phi(u)\rightarrow\infty$ as $\|u\|\rightarrow\infty$, i.e. $\Phi$  is coercive. Also by similar argument in \cite{Avci}, $\Phi'$ has an inverse
mapping $ (\Phi')^{-1}: W^* \to W$ which is continuous.
Additionally, functional $\Psi$ is a continuously G\^ateaux
differentiable functional and from $(F1)$ whose G\^ateaux derivative
is compact. The solutions of the equation
$I'_{\bar{\lambda}}=\Phi'-\bar{\lambda} \Psi'=0$ are exactly the
solutions for problem \eqref{1}, by Lemma \ref{l}.
 Hence, to prove our result, it is enough to apply Theorem \ref{t1}.\\
Let us define the function $\overline{v}$
$:\mathbb{Z}[0,T+1]\rightarrow\mathbb{R}$ belonging to $W$ by
\begin{equation}\label{v}
\bar{v}(t)=
\begin{cases}
d,\quad\quad\quad\quad k\in [1,T],\\
0, \quad\quad\quad\quad k= 0,T+1,
\end{cases}
\end{equation}
$$r_1=\frac{1}{p^+}\left({c_1K}\right)^{p^+},$$  and
$$r_2=\frac{1}{p^+}\left({c_2K}\right)^{p^+}.$$
 By \eqref{in2}, $r_1,r_2<\frac{1}{p^+}$
 and
 \begin{equation}\label{v-0}
  \Psi(\bar{v})=\sum_{k=1}^{T} F(k,\bar{v}(k))=\sum_{k=1}^{T}
  F(k,d),
 \end{equation}

  \begin{equation}\label{v-1}
  \Phi(\bar{v})=\frac{w(0)}{p(0)}d^{p(0)}+\frac{w(T)}{p(T)}d^{p(T)}+\sum_{k=1}^{T}\frac{q(k)}{p(k)}d^{p(k)}.
  \end{equation}
 By \eqref{in2}, one can conclude that $d\in(0,1)$, therefore $\frac{d^{p^+}}{p^+}A<\Phi(\bar{v})<\frac{d^{p^-}}{p^-}A$, hence by again \eqref{in2}, $r_1<\Phi(\bar{v})<r_2$.
Let be $u\in \Phi^{-1}(-\infty,r_i)$, $i=1,2$
 for all $u\in W$, then one has $\max_{ k\in[1,T]}|u(k)|\leq c_i$, $i=1,2$. Indeed, by
\eqref{1-3}, $\varphi(u)<p^+\Phi(u)<r_ip^+<1$, this means that
$\|u\|<1$, and therefore bearing in mind \eqref{x1} and \eqref{in},
\begin{equation}\label{5-1}
    K_0^{p^+}\|u\|^{p^+}<\|u\|_+^{p^+}<\varphi(u)<r_ip^+,
\end{equation}
hence, $\max_{ k\in[1,T]}|u(k)|\leq(2T+2)^{\frac{p^--1}{p^-}}\|u\|<(2T+2)^{\frac{p^--1}{p^-}}\frac{(r_ip^+)^{\frac{1}{p^+}}}{K_0}=c_i$.
Therefore,
$$\sup_{u\in \Phi^{-1}(-\infty,r_i)}\Psi(u)=\sup_{u\in \Phi^{-1}(-\infty,r_i)}\sum_{k=1}^{T} F(k,u(k))\leq \sum_{k=1}^{T}\max_{|\xi|\leq  c_i}
 F(k,\xi),\quad i=1,2.$$
Thus,
\begin{eqnarray*}
0\leq\beta(r_1,r_2)&\leq& \frac{\sup_{u\in \Phi^{-1}(-\infty,r_2)}\Psi(u)-\Psi(\bar{v})}{\frac{1}{p^+}\left({c_2K}\right)^{p^+}-\Phi(\bar{v})}\nonumber\\
&\leq&\frac{\sum_{k=1}^{T}\max_{|\xi|\leq  c_2} F(k,\xi)-\sum_{k=1}^{T} F(k,d)}{\frac{1}{p^+}\left({c_2K}\right)^{p^+}-\frac{d^{p^-}}{p^-}A}\nonumber\\
&=&a_d(c_2).
\end{eqnarray*}
\noindent On the other hand, one has\\
\begin{eqnarray*}
\rho_1(r_1,r_2)&\geq&\frac{\Psi(\bar{v})-\sup_{u\in\Phi^{-1}(-\infty,r_1)}\Psi(u)}{\Phi(\bar{v})-r_1}\nonumber\\
&\geq&\frac{\sum_{k=1}^{T} F(k,d)-\sum_{k=1}^{T}\max_{|\xi|\leq  c_1} F(k,\xi)}{\Phi(\bar{v})-\frac{1}{p^+}\left({c_1K}\right)^{p^+}}\nonumber\\
&\geq&\frac{\sum_{k=1}^{T} F(k,d)-\sum_{k=1}^{T}\max_{|\xi|\leq  c_1} F(k,\xi)}{\frac{d^{p^-}}{p^-}A-\frac{1}{p^+}\left(c_1K\right)^{p^+}}\nonumber\\
&=&a_d(c_1).
\end{eqnarray*}
Hence, from Assumption $(F2)$, we get
$\beta(r_1,r_2)<\rho_1(r_1,r_2)$. Therefore, owing to Theorem
\ref{t1}, for each $\lambda\in
]\frac{1}{a_d(c_1)},\frac{1}{a_d(c_2)}[$, the functional $I_\lambda$
admits one critical point $u_0 \in W$ such that $r_1<\Phi(u_0)<r_2$.
Hence, the proof is complete.
\end{proof}
\begin{corollary}\label{c1}
If $u_0$ be the ensured solution  in the conclusions of
Theorems \ref{t2}, then
$$\left(\frac{p^-}{p^+}\right)^{\frac{1}{p^-}}\left(c_1K\right)^{\frac{p^+}{p^-}}<\|u_0\|<c_2(2T+2)^{\frac{1-p^-}{p^-}}.$$
\end{corollary}
\begin{proof}
Since $\Phi(u_0)<r_2$, taking into account \eqref{5-1}, one can conclude that  $\|u_0\|<c_2(2T+2)^{\frac{1-p^-}{p^-}}$. On the other hand, by \eqref{x1}, since
$\varphi(u_0)<1$,
$$r_1<\Phi(u_0)<\frac{1}{p^-}\varphi(u_0)<\frac{\|u_0\|^{p^-}}{p^-}\Rightarrow \frac{1}{p^+}\left({c_1K}\right)^{p^+}<\frac{\|u_0\|^{p^-}}{p^-}.$$
That follows assertion.
\end{proof}
 We now present an example to illustrate the result of Theorem
\ref{t2}.
\begin{example}\label{e3.3}
Let $T=10$, $p(k)=\frac{2}{11}k+3$, $q(k)=2^{k}$,
$w(k)=e^{k(10-k)^2}$, $\alpha(k)=2$ and
$f(k,x)=e^{(k+2)(k-13)}\frac{x}{(x^2+10^{-11})^2}$ for
$k=1,2,3,...,10$ and $x\in\mathbb{R}$. Hence $p^-=3$, $p^+=5$,
$\alpha^+=2$,
  $A=2^{11}$ and $K=22^{-\frac{4}{5}}e^{-\frac{98}{5}}$ and $F(k,x)=\frac{10^{11}}{2}e^{(k+2)(k-13)}\frac{x^2}{x^2+10^{-11}}$. Put $c_0=0.000012$, $c_1=10^{-9}$ and
  $c_2=10^{9}$ and $d=10^{-5}$.
  Simple calculations show that
$$
F(k,t)\leq 0.000012(1+|t|^{2}),  \ \ \ \ \ \forall(k,t)\in
\mathbb{Z}[1,10]\times\mathbb{R}.
$$
  \begin{eqnarray*}
    a_d(c_1) &=& \frac{\sum^{10} _{k=1} F(k,c_1)-\sum^{10}_{k=1} F(k,d)}{\frac{1}{p^+}\left({c_1K}\right)^{p^+}-\frac{d^{p^-}}{p^-}A} \\
     &=&\frac{ \{\frac{10^{11}}{2}\frac{\left(10^{-9}\right)^2}{\left(10^{-9}\right)^2+10^{-11}}-
     \frac{10^{11}}{2}\frac{\left(10^{-5}\right)^2}{\left(10^{-5}\right)^2+10^{-11}}\}\sum^{10} _{k=1}e^{(k+2)(k-13)}}{\frac{1}{5}\left({10^{-9}\times 22^{-\frac{4}{5}}e^{-\frac{98}{5}}}\right)^{5}-\frac{\left(10^{-5}\right)^{3}}{3}2^{11}} \\
      &=&\frac{\{\frac{10^{-18}-10^{-10}}{10^{-28}+10^{-29}+10^{-21}+10^{-22}}\}\sum^{5} _{k=1}e^{(k+2)(k-13)}}{\frac{1}{5}\{10^{-45} 22^{-4}e^{-98}\}-\frac{1}{3}\{10^{-15}2^{11}\}} \\
    &=&30898916/775,
  \end{eqnarray*}
   and
    \begin{eqnarray*}
    a_d(c_2) &=& \frac{\sum^{10} _{k=1} F(k,c_2)-\sum^{10}_{k=1} F(k,d)}{\frac{1}{p^+}\left({c_2K}\right)^{p^+}-\frac{d^{p^-}}{p^-}A} \\
     &=&\frac{\{\frac{10^{11}}{2}\frac{\left(10^{9}\right)^2}{\left(10^{9}\right)^2+10^{-11}}-
     \frac{10^{11}}{2}\frac{\left(10^{-5}\right)^2}{\left(10^{-5}\right)^2+10^{-11}}\}\sum^{10} _{k=1}e^{(k+2)(k-13)}}{\frac{1}{5}\left({10^{9}\times 22^{-\frac{4}{5}}e^{-\frac{98}{5}}}\right)^{5}-\frac{\left(10^{-5}\right)^{3}}{3}2^{11}} \\
      &=&\frac{\{\frac{10^{18}-10^{-10}}{10^{8}+10^{7}+10^{-21}+10^{-22}}\}\sum^{5} _{k=1}e^{(k+2)(k-13)}}{\frac{1}{5}\{10^{45} 22^{-4}e^{-98}\}-\frac{1}{3}\{10^{-15}2^{11}\}} \\
    &=&0/009.
  \end{eqnarray*}
  Therefore the conditions $F1$ and $F2$ hold.
 Then, by Theorem \ref{t2}, for every $\lambda\in]0.000000033,111[$ the
problem
$$\begin{cases}
-\Delta(e^{k(10-k)^2}|
\Delta u(k-1)|^{p(k-1)-2}\Delta u(k-1))+2^{k}| u(k)|^{p(k)-2} u(k)=\lambda \left(\frac{ u(k)^2e^{(k+2)(k-13)}}{(u(k)^2+10^{-11})^2}\right), \\
u(0)=u(11)=0,
\end{cases}$$
for every $k \in [1,10]$, has at least one non-trivial solution $u_0$ that by Corollary \ref{c1},
$\frac{0.6^{\frac{1}{3}}}{10^{15}\times22^{\frac{4}{3}}\times e^{\frac{98}{3}}}<||u_0||<\frac{10^9}{22^{\frac{2}{3}}}$.
\end{example}
 Here we point out an immediate consequence of Theorem \ref{t2} as
 follows.
\begin{theorem}\label{t3}
 Assume the condition $(F1)$ holds and assume that there exist two positive constants $c$ and $d$ with
$Ap^+d^{p^-}<K^{p^+}p^-c^{p^+}<p^-$
such that\\\\
$(F3)$
$\sum^T _{k=1} \max_{ |\xi|\leq c} F(k,\xi)<\frac{p^-(cK)^{p^+}}{p^+d^{p^-}A}\sum^T_{k=1} F(k,d).$\\\\
Then, for each
$$\lambda\in \left]\frac{\frac{d^{p^-}}{p^-}A}{\sum^T_{k=1} F(k,d)},\frac{\frac{1}{p^+}\left({cK}\right)^{p^+}-\frac{d^{p^-}}{p^-}A}{\sum^T _{k=1} \max_{ |\xi|\leq c} F(k,\xi)-\sum^T_{k=1} F(k,d)}\right[,$$ the problem \eqref{1} has at least one non-trivial
solution $u_0\in W$ such that $||u_0||_\infty < c$.
\end{theorem}
\begin{proof}
By applying Theorem \ref{t2} and picking
$c_1=0$, $c_2=c$ the conclusion follows at once. Indeed, owing to our assumptions, one has
\begin{eqnarray*}a_{d}(c)&=&\frac{\sum^T _{k=1} \max_{ |\xi|\leq c} F(k,\xi)-\sum^T_{k=1} F(k,d)}{\frac{1}{p^+}\left({cK}\right)^{p^+}-\frac{d^{p^-}}{p^-}A}\\
&<&\frac{\sum^T_{k=1} F(k,d)}{\frac{d^{p^-}}{p^-}A}\\&=&a_{d}(0).
\end{eqnarray*}
Hence, Theorem \ref{t2} along with \eqref{4} and Corollary \ref{c1}, ensures the conclusion.
\end{proof}
\begin{remark}\label{r1}
If $f$ is non-negative, then, by consideration \cite[Theorem 2.2]{BC1}, the obtained solution $u_0$ in the conclusions of
Theorems \ref{t2} and \ref{t3} is non-negative. If $f(k,0)=0$ for all $k\in[0,T]$, owing to
\cite[Theorem 2.3]{BC1}, the obtained solution $u_0$ is positive
(see \cite[Remark 2.1]{BC1}).
\end{remark}
Next, we consider the  following problem, as a special case of the problem \eqref{1},
 \begin{equation}
\label{e3.1}
\begin{cases}
-\Delta(w(k-1)|
\Delta u(k-1)|^{p(k-1)-2}\Delta u(k-1))+q(k)| u(k)|^{p(k)-2}u(k)=\lambda \beta(k)g(u(k)),\\
u(0)=u(T+1)=0,
\end{cases}
\end{equation}
for any $k \in [1,T]$, where $\beta: [1,T]\to\mathbb{R}$ is a
nonnegative function and $g\in C(\mathbb{R},\mathbb{R})$ is a
continuous function. Put $G(t)=\int_{0}^tg(\xi)d\xi$ for all
$t\in\mathbb{R}$.
\begin{theorem}\label{t3}
Assume the condition $(F1)$ holds and assume that there exist two
positive constants $c$ and $d$ with
$Ap^+d^{p^-}<K^{p^+}p^-c^{p^+}<p^-$
such that\\\\
$(F4)$
$\max_{ |\xi|\leq c}G(\xi) <\frac{p^-(cK)^{p^+}}{p^+d^{p^-}A}G(d).$\\\\
Then, for each
$$\lambda\in \left]\frac{\frac{d^{p^-}}{p^-}A}{G(d)\sum^T_{k=1}\beta(k)},\frac{\frac{1}{p^+}\left({cK}\right)^{p^+}-\frac{d^{p^-}}
{p^-}A}{ [\max_{ |\xi|\leq c} G(\xi)-G(d)]\sum^T
_{k=1}\beta(k)}\right[,
$$ the problem \eqref{e3.1} has at least one non-trivial solution
$u_0\in W$ such that $||u_0||_\infty <c$.
\end{theorem}
\begin{proof}
Again, by applying Theorem \ref{t2} and picking $c_1=0$ and $c_2=c$ we have the conclusion. Indeed, owing to our assumptions, one has
\begin{eqnarray*}a_{d}(c)&=&\frac{\sum^T _{k=1}\beta(k)\left[\max_{ |\xi|\leq c} G(\xi) -G(d)\right] }{\frac{1}{p^+}\left({cK}\right)^{p^+}-\frac{d^{p^-}}{p^-}A}\\
&<&\frac{G(d)\sum^T_{k=1}\beta(k)
}{\frac{d^{p^-}}{p^-}A}\\&=&a_{d}(0).
\end{eqnarray*}
Thus, considering Theorem \ref{t2}, \eqref{4} and Corollary \ref{c1}, we obtain the desired conclusion.
\end{proof}
We now proceed with the proof of Theorem \ref{t1.1}.\\\\
{\bf PROOF OF THEOREM \ref{t1.1}}: This follows from Theorem \ref{t3} at once, by letting $p(k)=k+3$, $\alpha(k)=2$ and $w(k)=q(k)=\beta(k)=1$
  for every $k \in [1,T]$.\\

Here, we present the following example to illustrate the result of
Theorem \ref{t1.1}.


\begin{example}
Consider the problem
\begin{equation}\label{e3.4}
\left\{\begin{array}{ll}
-\Delta(|\Delta u(k-1)|^{k}\Delta u(k-1))+|u(k)|^{k+1}u(k)=\frac{\lambda}{ (400u(k))^2+1}, \quad k \in [1,10], \\
u(0)=u(11)=0.
\end{array}\right.
\end{equation}
Put $T=10$ and select  $d=0.1$, $c=17.1$ and $c_0=0.0039$ that
satisfying \eqref{1-1-1}, growth condition and \eqref{1-1-2},  that
is $d^{3}<\frac{1}{56\times22^{13}}c^{14}<\frac{1}{56}$ and
$\frac{1}{400}\arctan(400t)<c_0(1+t^2)$, for every $t\in\mathbb{R}$
and
$$\frac{\arctan(400c)}{c^{14}}<(\frac{1}{56\times22^{13}})\frac{\arctan(400d)}{d^{3}},$$
respectively, then for each $\lambda\in
\left]0.1035061724,67.87674577\right[$ the problem \eqref{e3.4} has
at least one  non-trivial  solution $u_{0}\in \{u : [0,10]\to
\mathbb{R} : u(0) = u(11) =0 \}$, that by Remark \ref{r1} is
non-negative.
 \end{example}
Now we state the second main result of the paper. We will apply
Theorem \ref{t1-1}. To do so, we provide the following theorem.
 \begin{theorem}\label{t4}
 Assume the condition $(F1)$ holds and
\begin{itemize}
\item[(F5)] There exist constants $d,c_{3}>0$ with $\frac{1}{\sqrt[p^+]{A}}>d>c_{3}\frac{K}{\sqrt[p^+]{A}}$ such that
\begin{equation*}
\frac{p^{+}}{(c_{3}K)^{p^{+}}}Tc_{0} (1+
\max\{c_{3}^{\alpha^{+}},c_{3}^{\alpha^{-}}\})<\hat{d}^{-1}\sum_{k=1}^{T}F(k,d),
\end{equation*}%
where
$\hat{d}=\frac{w(0)d^{p(0)}}{^{p(0)}}+\frac{w(T)d^{p(T)}}{^{p(T)}}+\sum_{k=1}^{T}\frac{q(k)}{p(k)}d^{p(k)}$.
\end{itemize}

Then for each $\lambda \in
\Lambda_{d}:=\left]\frac{1}{\hat{d}^{-1}\sum_{k=1}^{T}F(k,d)},+\infty\right[
$, the problem \eqref{1} admits at least one nontrivial weak
solution.
\end{theorem}
\begin{proof}
As mentioned in the proof of Theorem \ref{t1}, the regularities of
$\Phi$ and $\Psi$ hold. Let us define the function $\overline{v}$
$:\mathbb{Z}[0,T+1]\rightarrow\mathbb{R}$ belonging to $W$ by the
formula \eqref{v}. Then from \eqref{v-1} we deduce that
\begin{equation*}
\Phi( \overline{v})>\frac{A}{p^{+}}d^{p^+}.
\end{equation*}
Let fix $r=\frac{(c_{3}K)^{p^{+}}}{p^{+}}$. Since
$d^{p^+}>\frac{(c_{3}K)^{p^+}}{A}$, we get $\Phi(\overline{v})>r$.
On the other hand, by Lemma \ref{L1}, we have that $\Phi$ is bounded
on $W$. Therefore, since $\overline{v}\in W$ and $\inf_{u\in
W}\Phi(u)=\Phi(0)=0$ we obtain
\begin{equation*}
\inf_{u\in W}\Phi(u)<r<\Phi(\overline{v})<\sup_{u\in W}\Phi(u).
\end{equation*}
For each $u\in \Phi ^{-1}]-\infty,r[$, by similar argument for
obtaining \eqref{5-1},  we have
\begin{equation*}
\|u\|< \frac{(rp^{+})^{\frac{1}{p^{+}}}}{K_{0}},
\end{equation*}
which leads us, by \eqref{4}, to
\begin{equation*}
\max_{k\in [1,T]}|u(k)|\leq
(2T+2)^{\frac{p^{-}-1}{p^{-}}}\|u\|<(2T+2)^{\frac{p^{-}-1}{p^{-}}}\frac{(rp^{+})^{\frac{1}{p^{+}}}}{K_{0}}=:c_{3}.
\end{equation*}
Therefore, from the condition $(F1)$ and $(F5)$, it reads
\begin{eqnarray}\label{3.8}
\sup_{u\in \Phi ^{-1}(-\infty,r)}\Psi(u)&\leq&
\sum_{k=1}^{T}\max_{|\xi|\leq c_{3}} F(k,\xi) \leq
\sum_{k=1}^{T}\max_{|\xi|\leq c_{3}} c_{0}(1+|\xi|^{\alpha(k)})
\nonumber \\
&< & Tc_{0} (1+
\max\{c_{3}^{\alpha^{+}},c_{3}^{\alpha^{-}}\})<r\frac{\Psi(\overline{v})}{\Phi(\overline{v})}.
\end{eqnarray}
Considering \eqref{3.8}, we obtain
\begin{eqnarray*}
\rho_{2}(r)&=&\sup_{v\in \Phi
^{-1}(r,\infty)}\frac{\Psi(v)-\sup_{u\in \Phi
^{-1}(-\infty,r)}\Psi(u)}{\Phi(v)-r}\\
&\geq& \frac{\Psi(\overline{v})-\sup_{u\in \Phi
^{-1}(-\infty,r)}\Psi(u)}{\Phi(\overline{v})-r}\geq
\frac{\Psi(\overline{v})-r\frac{\Psi(\overline{v})}{\Phi(\overline{v})}}{\Phi(\overline{v})-r}=
\frac{\Psi(\overline{v})}{\Phi(\overline{v})}>0.
\end{eqnarray*}
Let us proceed with the coercivity of $I_{\lambda }$. To this end,
let $u\in W$ such that $\Vert u\Vert\rightarrow +\infty $. Then,
without loss of generality, we can assume that $\Vert u\Vert>1$.
Then from \eqref{x1}, \eqref{4} and condition $(F1)$, it reads
\begin{eqnarray*}
I_{\lambda}(u) &\geq& \frac{\|u\|^{p^{-}}}{p^{+}}-\lambda \sum_{k=1}^{T}c_{0}(1+|u(k)|^{\alpha(k)}) \\
&\geq& \frac{\|u\|^{p^{-}}}{p^{+}}-\lambda \sum_{k=1}^{T}c_{0}(1+(\max_{k\in [1,T]}|u(k)|)^{\alpha(k)}) \\
&\geq& \frac{\|u\|^{p^{-}}}{p^{+}}-\lambda \sum_{k=1}^{T}c_{0}(1+((2T+2)^{\frac{p^{-}-1}{p^{-}}}\|u\|)^{\alpha(k)}) \\
&\geq& \frac{\|u\|^{p^{-}}}{p^{+}}-\lambda Tc_{0}(2T+2)^{\frac{(p^{-}-1)\alpha^{+}}{p^{-}}}\|u\|^{\alpha^{+}}-\lambda Tc_{0}, \\
\end{eqnarray*}%
that due to $\alpha^+<p^-$, it follows that $I_{\lambda }$ be
coercive. Consequently, all assumptions of Theorem 2.2 are verified.
Therefore, for each $\lambda \in \Lambda _{d}$, the problem
\eqref{1} admits at least one nontrivial weak solution.
\end{proof}
In the following we give a corollary which is based on Theorem 2.4
of \cite {Bonanno}.

\begin {corollary}\label{c10} Assume that $(F1)$ and $(F5)$ holds. Then for each\\ $\lambda \in \Lambda _{r}:=\left]\frac{1}{%
\hat{d}^{-1}\sum_{k=1}^{T}F(k,d)},\frac{r}{Tc_{0} (1+
\max\{c_{3}^{\alpha^{+}},c_{3}^{\alpha^{-}}\})}\right[$, the problem
\eqref{1} admits at least three distinct weak solutions.
\end {corollary}

\begin{proof} So far we have already obtained that $\Phi$ is a continuously G\^{a}teaux differentiable,
 coercive and sequentially weakly lower semi-continuous functional whose G\^{a}teaux derivative admits a
  continuous inverse on $W^{\ast }$, and $\Psi$ is continuously G\^{a}teaux differentiable functional
   whose G\^{a}teaux derivative is compact, and $\inf_{x\in X}\Phi (u)=\Phi(0)=\Psi(0)=0$.
    Moreover, since $I_{\lambda}$ is coercive on $\Lambda _{d}$, it is coercive on $\Lambda _{r}$ as well because of the
     relation $\Lambda _{r}\subseteq\Lambda _{d}$. The rest of the proof is quite similar to that of Theorem 3.8.
     However, some remarks are in order. Since, apparently
      $$\Phi(\overline{v})=\hat{d}=\frac{w(0)d^{p(0)}}{^{p(0)}}+\frac{w(T)d^{p(T)}}{^{p(T)}}+\sum_{k=1}^{T}\frac{q(k)}{p(k)}d^{p(k)}$$
      and
       $$\Psi(\overline{v})=\sum_{k=1}^{T}F(k,d),$$
       it reads, from (3.8) and (F5), that
\begin{equation*}
\sup_{u\in \Phi ^{-1}(-\infty,r)}\Psi(u)< Tc_{0} (1+
\max\{c_{3}^{\alpha^{+}},c_{3}^{\alpha^{-}}\})\leq r
\frac{\Psi(\overline{v})}{\Phi(\overline{v})}.
\end{equation*}
Overall, all assumptions of Theorem 2.4 of \cite {Bonanno} are
verified. Therefore, for each $\lambda \in \Lambda _{r}$, the
functional $I_{\lambda}$ admits at least three distinct critical
points that are weak solutions of Problem (1.1).
\end{proof}
Here, we present the following example to illustrate the result of
Theorem \ref{t4}.

\begin{example} Let all assumptions in Example \ref{e3.3} hold, that
is, $T=10$, $p(k)=\frac{2}{11}k+3$, $q(k)=2^{k}$,
$w(k)=e^{k(10-k)^2}$, $\alpha(k)=2$ and
$f(k,x)=e^{(k+2)(k-13)}\frac{x}{(x^2+10^{-11})^2}$ for
$k=1,2,3,...,10$ and $x\in\mathbb{R}$. Hence $p^-=3$, $p^+=5$,
$\alpha^+=2$,
  $A=2^{11}$ and $K=22^{-\frac{4}{5}}e^{-\frac{98}{5}}$ and $F(k,x)=\frac{10^{11}}{2}e^{(k+2)(k-13)}\frac{x^2}{x^2+10^{-11}}$.
   Put $c_0=0.000012$, $c_3=0.05$ and $d=0.0000000005$ satisfying  $\frac{1}{\sqrt[p^+]{A}}>d>c_{3}\frac{K}{\sqrt[p^+]{A}}$. Simple
    calculations with Maple software  show that
$$
F(k,t)\leq 0.000012(1+|t|^{2}),  \ \ \ \ \ \forall(k,t)\in
\mathbb{Z}[1,10]\times\mathbb{R},
$$
$$
\frac{p^{+}}{(c_{3}K)^{p^{+}}}Tc_{0} (1+
\max\{c_{3}^{\alpha^{+}},c_{3}^{\alpha^{-}}\})\simeq2.086867833\times10^{13},$$
and
$$\hat{d}^{-1}\sum_{k=1}^{T}F(k,d)\simeq1.338709020\times10^{16},$$
  Therefore the conditions $F1$ and $F5$ hold.
 Then, by Theorem \ref{t4}, for every $\lambda\in]7.469883186\times10^{-17},\infty[$ the
problem
$$\begin{cases}
-\Delta(e^{k(10-k)^2}|
\Delta u(k-1)|^{p(k-1)-2}\Delta u(k-1))+2^{k}| u(k)|^{p(k)-2} u(k)=\lambda \left(\frac{ u(k)^2e^{(k+2)(k-13)}}{(u(k)^2+10^{-11})^2}\right), \\
u(0)=u(11)=0,
\end{cases}$$
for every $k \in [1,10]$, has at least one non-trivial solution and
by Corollary \ref{c10}, for every
$\lambda\in]7.469883186\times10^{-17},4.791870305\times10^{-14}[$
the above problem has at least three solutions.
 \end{example}

%

\end{document}